\documentclass[11pt]{article}
\usepackage{amsthm,amssymb,amsmath, amsfonts}
\usepackage{enumerate}
\usepackage{hyperref}
\hypersetup{
    pdftitle=   {A remark on the paper ``Properties of intersecting families of ordered sets''
by O. Einstein},
   pdfauthor=  {Sang-il Oum and Sounggun Wee}
}
\usepackage{authblk}
\newtheorem{THM}{Theorem}
\newtheorem{CON}{Conjecture}
\newtheorem*{CLAIM}{Claim}

\newcommand\abs[1]{\lvert #1\rvert}
\newcommand\F{\mathbb F}
\begin{document}
\title{A remark on the paper ``Properties of intersecting families of ordered sets''
by O.~Einstein}
\author{Sang-il Oum\thanks{\texttt{sangil@kaist.edu}}
\thanks{Supported by the National Research Foundation of Korea (NRF) grant funded by the Korea government (MSIP) (No. NRF-2017R1A2B4005020).} }
\author{Sounggun Wee\thanks{\texttt{tjdrns2704@kaist.ac.kr}}}

\affil{Department of Mathematical Sciences, KAIST,  Daejeon, South Korea.}

\date\today
\maketitle

\begin{abstract}
  O.~Einstein (2008) proved Bollob\'as-type
  theorems on intersecting families of ordered sets of finite sets and subspaces.
Unfortunately, we report that the proof of a theorem on
ordered sets of subspaces had a mistake. We prove two weaker variants.
\end{abstract}
\section{Introduction}
The following theorem generalizing the theorem of Bollob\'as~\cite{Bollobas1965}
is well known and proved by using the wedge product method (see~\cite{BF1992}).
\begin{THM}[Lov\'asz~\cite{Lovasz1977}; skew version]\label{thm:lovasz}
  Let $a$, $b$ be positive integers.
  Let $U_1,U_2,\ldots,U_m$, $V_1,V_2,\ldots,V_m$ be subspaces satisfying the following:
  \begin{enumerate}[(i)]
  \item $\dim U_i\le a$ and $\dim V_i\le b$ for all $i=1,2,\ldots,m$.
  \item $U_i\cap V_i=\{0\}$ for all $i=1,2,\ldots,m$.
  \item $U_i\cap V_j\neq\{0\}$ for all $1\le i<j\le m$.
  \end{enumerate}
  Then $m\le \binom{a+b}{a}$.
\end{THM}
Ori Einstein~\cite{Einstein2008} published a paper on a generalization of the above theorem and its consequence on finite sets by Frankl~\cite{Frankl1982}.
We will show that his proof of Theorem 2.7 in \cite{Einstein2008} is incorrect and so we state it as a conjecture.
\begin{CON}[Theorem 2.7 of \cite{Einstein2008}]\label{con}
  Let $\ell_1$, $\ell_2$, $\ldots$, $\ell_k$ be positive integers.
  Let $U$ be a linear space over a field $\F$.
  Consider the following matrix of subspaces:
\[
  \begin{matrix}
    U_{11}&U_{12}&\cdots&U_{1k}\\
    U_{21}&U_{22}&\cdots&U_{2k}\\
    \cdots & \cdots & \vdots & \cdots\\
    U_{m1}&U_{m2}&\cdots&U_{mk}
  \end{matrix}
\]
If these subspaces satisfy:
\begin{enumerate}[(i)]
\item for every $1\le j\le k$, $1\le i\le m$, $\dim U_{ij}\le \ell_j$;
\item for every fixed $i$, all subspaces $U_{ij}$ are pairwise disjoint;
\item for each $i<i'$, there exist some $j<j'$ such that $U_{ij}\cap
  U_{i'j'}\neq \{0\}$;
\end{enumerate}
then \[m\le \frac{(\sum_{r=1}^k \ell_r)!}{\prod_{r=1}^k \ell_r!}.\]
\end{CON}
Here is the overview of this note.
In the next section, we will sketch the reason why the proof of Theorem 2.7 in \cite{Einstein2008} is
incorrect and present a weaker theorem (Theorem~\ref{thm:first}) obtained by tightening condition (ii).
In Section~\ref{sec:variant2}, 
we prove another weaker theorem (Theorem~\ref{thm:second}), 
by providing a weaker upper bound for $m$ instead of modifying any assumptions.
Section~\ref{sec:threshold} will discuss the threshold versions.

\section{The mistake and its first remedy}
Let us first point out the mistake in the proof of Conjecture~\ref{con} in \cite{Einstein2008}. As it is typical in the wedge product method, 
we take $v_i=\bigwedge_{j=1}^{k-1} \wedge T_j(U_{ij})$ 
and $w_i=\bigwedge_{j=1}^{k-1} \bigwedge_{r=j+1}^k \wedge T_j(U_{ir})$
for some linear transformations $T_1$, $T_2$, $\ldots$, $T_{k-1}$.
Then the following claim is made:

\begin{CLAIM}[Page 41 in \cite{Einstein2008}]
  For every $i\le i'$, $v_i\wedge w_{i'}\neq0$ if and only if $i=i'$.
\end{CLAIM}
This claim is false in general.
For instance, if $U_{11}\cap (U_{12}+U_{13}+\cdots+U_{1,k-1})\neq \{0\}$, then 
$\wedge T_1(U_{11}) \wedge \bigwedge_{r=2}^k \wedge T_1(U_{1r})=0$
and therefore $v_1\wedge w_1'=0$.
The crucial mistake is that condition (ii) in Conjecture~\ref{con} does not imply
that $\dim (U_{i1}+U_{i2}+\cdots+U_{ik})=\sum_{j=1}^k \dim U_{ij}$.
(For instance the spans of 
$\left(\begin{smallmatrix}
  1\\
  0
\end{smallmatrix}\right)$, 
$\left(\begin{smallmatrix}
  0\\
  1
\end{smallmatrix}\right)$, and
$\left(\begin{smallmatrix}
    1\\
  1
\end{smallmatrix}\right)$ are pairwise disjoint and yet their sum has
dimension $2$ only.)

If $\dim (U_{i1}+U_{i2}+\cdots+U_{ik})=\sum_{j=1}^k \dim U_{ij}$,
then
the claim is true and so we can recover the following weaker theorem
by the proof in \cite{Einstein2008}.
\begin{THM}\label{thm:variant1}
  Let $\ell_1$, $\ell_2$, $\ldots$, $\ell_k$ be positive integers.
  Let $U$ be a linear space over a field $\F$.
  Consider the following matrix of subspaces:
\[
  \begin{matrix}
    U_{11}&U_{12}&\cdots&U_{1k}\\
    U_{21}&U_{22}&\cdots&U_{2k}\\
    \cdots & \cdots & \vdots & \cdots\\
    U_{m1}&U_{m2}&\cdots&U_{mk}
  \end{matrix}
\]
If these subspaces satisfy:
\begin{enumerate}[(i)]
\item for every $1\le j\le k$, $1\le i\le m$, $\dim U_{ij}\le \ell_j$;
\item for every fixed $i$,  $\dim (\sum_{j=1}^k
  U_{ij})=\sum_{j=1}^k \dim U_{ij}$;
\item for each $i<i'$, there exist some $j<j'$ such that $U_{ij}\cap
  U_{i'j'}\neq \{0\}$; 
\end{enumerate}
then \[m\le \frac{(\sum_{r=1}^k \ell_r)!}{\prod_{r=1}^k \ell_r!}.\]
 \end{THM}
Though Theorem~\ref{thm:variant1} is weaker than
Conjecture~\ref{con}, it allows us to recover Theorem 2.8 of
\cite{Einstein2008}.
\begin{THM}[Theorem 2.8 of \cite{Einstein2008}]\label{thm:first}
  Let $\ell_1$, $\ell_2$, $\ldots$, $\ell_k$ be positive integers.
  Consider the following matrix of sets:
\[
  \begin{matrix}
    A_{11}&A_{12}&\cdots&A_{1k}\\
    A_{21}&A_{22}&\cdots&A_{2k}\\
    \cdots & \cdots & \vdots & \cdots\\
    A_{m1}&A_{m2}&\cdots&A_{mk}
  \end{matrix}
\]
If these sets satisfy:
\begin{enumerate}[(i)]
\item for every $1\le j\le k$, $1\le i\le m$, $\abs{A_{ij}}\le \ell_j$;
\item for every fixed $i$, all sets $A_{ij}$ are pairwise disjoint;
\item for each $i<i'$, there exist some $j<j'$ such that $A_{ij}\cap
  A_{i'j'}\neq \emptyset$;
\end{enumerate}
then \[m\le \frac{(\sum_{r=1}^k \ell_r)!}{\prod_{r=1}^k \ell_r!}.\]
\end{THM}

Note that Theorem~\ref{thm:first} implies that Conjecture~\ref{con} is
true when $\ell_1=\ell_2=\cdots=\ell_k=1$.

\section{Second remedy}\label{sec:variant2}
Naturally we ask whether Conjecture~\ref{con} can be proven with some
upper bound on $m$. Here we show that this is possible, while
generalizing Theorem~\ref{thm:lovasz}.
\begin{THM}\label{thm:second}
  Under the same assumptions of Conjecture~\ref{con},  we have
  \[
    m\le \frac{
      \prod_{1\le a<b\le k} (\ell_a+\ell_b)!
    }{ (\prod_{r=1}^k \ell_r!)^{k-1}}.
  \]
\end{THM}
\begin{proof}
  We may assume that $\dim U_{ij}=\ell_j$ for all $i$, $j$ and $\F$ is infinite.
  Let $V=V_{1,2}\oplus V_{1,3}\oplus \cdots \oplus V_{1,k}\oplus\cdots\oplus
  V_{2,3}\oplus \cdots\oplus V_{k-1,k}= \bigoplus_{a=1}^{k-1}\bigoplus_{b=a+1}^k V_{a,b}$
  be a $\sum_{a=1}^{k-1}\sum_{b=a+1}^k (\ell_a+\ell_b)$-dimensional vector space over $\F$, decomposed into the direct sum of subspaces $V_{a,b}$, each of dimension
  $\ell_a+\ell_b$.
  By Corollary 3.14 of \cite{BF1992}, for all $i<j$, there exists a linear transformation $T_{ab}:U\to V_{a,b}$ such that for all $1\le i\le m$,
  $\dim T_{ab}(U_{ia}) = \ell_a$,
  $\dim T_{ab}(U_{ib}) = \ell_b$,
  and $\dim T_{ab}(U_{ia})\cap T_{ab}(U_{jb})=\dim U_{ia}\cap U_{jb}$ for all $1\le i,j\le m$.
  Finally, for each $1\le i\le m$,
  let $v_i= \bigwedge_{a=1}^{k-1}\bigwedge_{b=a+1}^k \wedge T_{ab}(U_{ia})$
  and $w_i= \bigwedge_{a=1}^{k-1}\bigwedge_{b=a+1}^k \wedge T_{ab}(U_{ib})$.

  We claim that for $i\le i'$, $v_i\wedge w_{i'}\neq0$ if and only if $i=i'$.
  If $i<i'$, then there exist $1\le j<j'\le k$ such that $U_{ij}\cap U_{i'j'}\neq\{0\}$. By the choice of $T_{jj'}$, $T_{jj'}(U_{ij})\cap T_{jj'}(U_{i'j'})\neq\{0\}$ and so $(\wedge T_{jj'}(U_{ij}))\wedge
  (\wedge T_{jj'}(U_{i'j'}))=0$, which implies that $v_i\wedge w_{i'}=0$.
  If $i=i'$, then $v_i\wedge w_{i'}$ is the wedge product of disjoint subspaces
  and so $v_i\wedge w_{i'}\neq 0$.

  Therefore $v_1,v_2,\ldots,v_m$ are linearly independent in the space $\bigwedge_{a=1}^{k-1}\bigwedge_{b=a+1}^k \bigwedge^{\ell_a}V_{a,b}$, whose dimension is
  $\prod_{a=1}^{k-1}\prod_{b=a+1}^k \binom{\ell_a+\ell_b}{\ell_a}
  = \frac{\prod_{1\le a<b\le k}(\ell_a+\ell_b)!}
  {(\prod_{i=1}^k \ell_i!)^{k-1}}
  $. This proves that $m\le 
  \frac{\prod_{1\le a<b\le k}(\ell_a+\ell_b)!}
  {(\prod_{i=1}^k \ell_i!)^{k-1}}$.
\end{proof}

\section{Threshold versions}\label{sec:threshold}
The paper \cite{Einstein2008} uses Conjecture~\ref{con} to deduce the
threshold versions (Lemma 2.9 and Theorem 2.10) to generalize a result of
F\"uredi~\cite{Furedi1984}.
We do not know how to prove Lemma 2.9 and Theorem 2.10 of
\cite{Einstein2008} and so we leave them as conjectures. 
It is not clear how one can relax conditions in Lemma 2.9 and Theorem
2.10 of \cite{Einstein2008}, while avoiding ugly conditions from (ii)
of Theorem~\ref{thm:first}. 
(A necessary condition $\ell_i\ge t$ was missing in \cite{Einstein2008}.)
\begin{CON}[Lemma 2.9 of \cite{Einstein2008}]\label{con2}
  Let $\ell_1$, $\ell_2$, $\ldots$, $\ell_k$ be positive integers such
  that $\ell_i\ge t$ for all $i$.
  Let $U$ be a linear space over a field $\F$.
  Consider the following matrix of subspaces:
\[
  \begin{matrix}
    U_{11}&U_{12}&\cdots&U_{1k}\\
    U_{21}&U_{22}&\cdots&U_{2k}\\
    \cdots & \cdots & \vdots & \cdots\\
    U_{m1}&U_{m2}&\cdots&U_{mk}
  \end{matrix}
\]
If these subspaces satisfy:
\begin{enumerate}[(i)]
\item for every $1\le j\le k$, $1\le i\le m$, $\dim U_{ij}\le \ell_j$;
\item for every fixed $i$,  $\dim(U_{ij}\cap U_{ij'})\le t$;
\item for each $i<i'$, there exists some $j<j'$ such that $\dim (U_{ij}\cap U_{i'j'})>t$;
\end{enumerate}
then \[m\le \frac{[(\sum_{r=1}^k \ell_r)-kt]!}{\prod_{r=1}^k (\ell_r-t)!}.\]
\end{CON}
\begin{CON}[Theorem 2.10 of \cite{Einstein2008}]\label{con3}
  Let $\ell_1$, $\ell_2$, $\ldots$, $\ell_k$ be positive integers such
  that $\ell_i\ge t$ for all $i$.
  Consider the following matrix of sets:
\[
  \begin{matrix}
    A_{11}&A_{12}&\cdots&A_{1k}\\
    A_{21}&A_{22}&\cdots&A_{2k}\\
    \cdots & \cdots & \vdots & \cdots\\
    A_{m1}&A_{m2}&\cdots&A_{mk}
  \end{matrix}
\]
If these sets satisfy:
\begin{enumerate}[(i)]
\item for every $1\le j\le k$, $1\le i\le m$, $\abs{A_{ij}}\le \ell_j$;
\item for every $i$, $j$ and $j'$, $\abs{A_{ij}\cap A_{ij'}}\le t$;
\item for each $i<i'$, there exists some $j<j'$ such that 
$\abs{A_{ij}\cap  A_{i'j'}}>t$;
\end{enumerate}
then \[m\le \frac{[(\sum_{r=1}^k \ell_r)-kt]!}{\prod_{r=1}^k (\ell_r-t)!}.\]
\end{CON}

By using Theorem~\ref{thm:second}, we can prove the following weaker
variants of Conjectures~\ref{con2} and \ref{con3} by the same
reduction in \cite{Einstein2008}.
\begin{THM}
  Under the same assumptions of Conjecture~\ref{con2}, we have 
  \[
    m\le \frac{
      \prod_{1\le a<b\le k} (\ell_a+\ell_b-2t)!
    }{ (\prod_{r=1}^k (\ell_r-t)!)^{k-1}}.
  \]
\end{THM}
\begin{THM}
  Under the same assumptions of Conjecture~\ref{con3}, we have 
  \[
    m\le \frac{
      \prod_{1\le a<b\le k} (\ell_a+\ell_b-2t)!
    }{ (\prod_{r=1}^k (\ell_r-t)!)^{k-1}}.
  \]
\end{THM}
\section*{Acknowledgement}
The author would like to thank students attending the graduate course
on combinatorics in the spring semester of 2017;
they pointed out the difficulty, when they were given Conjecture~\ref{con}
for $k=3$ 
as a homework problem.
We also like to thank the anonymous referees for their helpful comments.


\end{document}